\documentclass[reqno, 11pt]{amsart}

\usepackage{amsthm, amsmath, amsfonts, amscd, amssymb, stmaryrd, mathrsfs}
\usepackage[breaklinks,colorlinks,citecolor=teal,linkcolor=teal,urlcolor=teal,hyperindex]{hyperref}
\usepackage{color}
\usepackage{tikz}
\usetikzlibrary{arrows} \usetikzlibrary{patterns}
\usepackage[all]{xy}
\usepackage[top=2.5cm, bottom=2.5cm, left=3cm, right=3cm]{geometry}
\newtheorem{theorem}{Theorem}[section]
\newtheorem{lemma}[theorem]{Lemma}

\newtheorem{conj}[theorem]{Conjecture}

\theoremstyle{definition}

\newtheorem{remark}[theorem]{Remark}

\numberwithin{equation}{section}

\title[]{Topological characterization and Hodge structures of some rationally elliptic projective fourfolds }

\author{Jianqiang Yang}
\address{Department of Mathematics, Honghe University,  Yunnan 661199, China}
\email{yangjianqiang@uoh.edu.cn}
\email{yangjq\_math@sina.com}

\date{}

\begin{document}

\begin{abstract}
In this paper, we consider the rationally elliptic projective fourfolds that are holomorphically embedded into the complex projective eight-space $\mathbb{P}^8$. It is proved that a simply-connected $\mathbb Q$-homological projective four-space $X\subset\mathbb{P}^8$ is biholomorphic to $\mathbb P^4$ by using Euler characteristic and Chern numbers formulae of the normal bundle for a holomorphic embedding $i:X \to\mathbb{P}^8$. During the process of proving the result, we incidentally discovered that a $\mathbb{Q}$-homological projective 4-space $X$ with Kodaira dimension $k(X) \neq 4$ is isomorphic to $\mathbb{P}^4$. This finding provides a positive answer to a question posed by Wilson in the case where the dimension $n=4$.
Using a similar approach,   we show that the Hodge conjecture holds for the rationally elliptic  fourfold $X \subset\mathbb{P}^8$, and the rationally elliptic  fourfold $X \subset\mathbb{P}^8$ has non-positive Hodge level.
\end{abstract}

\maketitle


\section{introduction}

 A simply-connected closed manifold $X$ is \emph{rationally elliptic}  (or \emph{of  elliptic homotopy type}) if
$$
\sum\limits_{k\geq 2}\dim(\pi_{k}(X)\otimes_{\mathbb Z }\mathbb Q)< +\infty.
$$
where $\pi_{k}(X)$ is the $k$-th homotopy group of $X$.
The rationally elliptic manifolds have nice properties and satisfy strong restrictions.
The K\"ahler manifolds have rich structures,
this may remind people of paying attention to the K\"ahler manifolds with elliptic homotopy type.
The characterization of rationally elliptic K\"ahler manifolds of dimension $\leq 3$ were studied by Amoros-Biswas (c.f.\cite{AB}).
In \cite{SY}, Su and the author obtained all possible Hodge structures of rationally elliptic  K\"ahler fourfolds and a partial characterization of rationally elliptic  K\"ahler manifold of  higher dimension.

In this paper, we are interested in rationally elliptic projective manifolds.
It is well-known that every $n$-dimensional projective manifold $X$ is holomorphically embedding into $\mathbb{P}^{2n+1}$ (c.f.\cite[P.173]{GH}).
Consider a general projection $p:X \to X'\subset \mathbb P^{2n}$ which is induced by a holomorphic embedding $X \to \mathbb{P}^{2n+1}$.
It is shown in \cite{GP} that we can  suitable choice of the general projection such that $X'$ has only finite ordinary double points as singularities. Since the  morphism $p$ is regular birational map, by \cite[Proposition 16.8]{HAR}, if the fiber $p^{-1}(q),q\in X'$ is  disconnected,then $q$ is a singular point. This indicates that
 such general projection $p$ is an immersion.  Using the Whitney trick (c.f.\cite[Subsection 7.3]{RA}), we can kill these ordinary double points of $X'$ and change the immersion $p$ into a smooth embedding.
From this viewpoint, since the rationally elliptic projective manifolds are simply connected and have strong restrictions, it is natural to ask  {\it whether a $n$-dimensional rationally elliptic projective manifold is holomorphically embedding into $\mathbb{P}^{2n}$? }
 The answer is positive for $n=1$,  and it is yet unknown for $n>1$.
Motivated by this question, we will further study the geometry and topology of rationally elliptic K\"ahler fourfolds which can holomorphically embed into $\mathbb{P}^8$.

\subsection{Topological characterization of $\mathbb{P}^4$}\label{subsec 1.1}

 The topological characterization of the complex projective space $\mathbb{P}^n$ is a long-standing problem.   Severi raised the question whether a complex surface homeomorphic to $\mathbb{P}^2$ has to be biholomorphic to $\mathbb{P}^2$.
The answer was affirmative by Yau (c.f.\cite{YST}).
In higher dimensional cases, it was solved by Hizebruch and Kodaira (c.f.\cite{HKK}), and Yau (c.f.\cite{YST})
 for K\"{a}hler manifolds.
 The next natural question is: {\it whether the hypothesis ``homeomorphism" on K\"{a}hler manifolds  can be relaxed?}
This is our main motivation.
As a matter of fact, there were many interesting works on this topic, such as  Debarre \cite{De}, Fujita \cite{F80}, Li \cite{LP},  Libgober and Wood \cite{LW}), Lanteri and Struppa \cite{LS} and so on.
The common feature, which they focused on, is that a K\"{a}hler manifold $X$ has the same integral cohomology ring as $\mathbb{P}^n.$
In this paper, ``the  same integral cohomology ring as $\mathbb{P}^n$" can be relaxed by  ``the  same rational cohomology ring as $\mathbb{P}^n$".

A \emph{$\mathbb Q$-homological projective space}  is a compact K\"ahler manifold $X$  with  the same Betti numbers as $\mathbb{P}^{n}$, i.e., $X$ has the same rational cohomology as $\mathbb{P}^n$ (c.f.\cite{Wi}).  The classification of the  $\mathbb Q$-homological projective space of dimension up to $3$ is well-known (cf.,\cite{PY07,Wi,YSK}). One of the famous work is the classification of all fake projective planes  by  Prasad and Yeung (c.f.\cite{PY07}).
It is shown in \cite{Wi} that the   $\mathbb Q$-homological projective four-space is either isomorphic to $\mathbb{P}^4$, or to another variety with pre-described invariants.    Furthermore, some interesting examples of non-simply connected $\mathbb Q$-homological projective four-spaces were constructed  by Prasad and Yeung (c.f.\cite{PY09}). In  \cite{YSK},  Yeung gives criteria for a fake projective four-space to be uniformized by the complex hyperbolic space of complex dimension four. So far the existence of simply-connected fake projective four-space is yet unknown.

In this paper, we focus on the simply-connected $\mathbb Q$-homological projective four-space.  It is the simplest  rationally elliptic K\"ahler manifold for Hodge structure.  From Lemma \ref{lem:RHPSB}, we prove that there is no simply connected $\mathbb Q$-homological projective four-space $X\subset \mathbb{P}^8$  with the first Chern class $c_{1}(X)<0$.
 It was raised by Wilson \cite{Wi} the question { \it whether a $\mathbb Q$-homological projective $n$-space $X$ with Kodaira dimension $k(X)\neq n$ is isomorphic to $\mathbb{P}^n$?}  In particular, for the case $n=4$ is still unsolved.
Based on the method of the proof of Lemma \ref{lem:RHPSB}, we conclude that if the first Chern class $c_{1}(X)>0$, the only   $\mathbb Q$-homological projective four-space $X$ is $\mathbb{P}^4$  (see Lemma \ref{lem:RHPSS}).
This gives a positive answer to the question of Wilson for $n=4$.
 More precisely, we obtain the following characterization of $\mathbb{P}^4$ by Betti numbers.

\begin{theorem}\label{thm:RHPS}
A simply-connected $\mathbb Q$-homological projective four-space $X\subset\mathbb{P}^8$  is biholomorphic to $\mathbb{P}^{4}$.
\end{theorem}

\subsection{The Hodge conjecture}\label{subsec 1.2}

The Hodge conjecture is a major unsolved problem in algebraic geometry and complex geometry.  Let
\[
Hdg^{2k}(X)=H^{2k}(X;\mathbb Q)\cap H^{k,k}(X)
\]
be the group of Hodge classes of degree $2k$ on  complex projective manifold $X$.

\begin{conj}[Hodge conjecture] \label{Conj} \cite{LG}
Let $X$ be a non-singular complex projective manifold.
Then every Hodge class on $X$ is  a linear combination with rational coefficients of the cohomology classes of complex  subvarieties (algebraic cycles) of $X$.
\end{conj}

This conjecture holds for $X$ of the dimension up to $3$, and there are few results on projective fourfold; for example, hypersurfaces with degree $d\leq 5$ (c.f.\cite{LG}), projective uniruled fourfold (c.f.\cite{CM}), varieties with the isomorphism map $cl_X: CH_{\ast}(X)\rightarrow H_{\ast}(X)$ (c.f.\cite[Example 19.1.11]{FU}). Recall that the odd Betti numbers are all zero  for rationally elliptic K\"ahler fourfold (c.f.\cite[Theorem 1.1]{SY}). It was proved in \cite[Corollary 1.2]{SY} that  for a rationally elliptic  projective fourfold, the Hodge conjecture is true, except for  the case that  the even dimensional  Hodge diamond is
\begin{align}\label{E hodge-diamond-1}
\begin{array}{ccccc}
& & 1 & & \\
& 0 & 1 & 0 & \\
0 & 0 & 2 & 0 & 0\\
& 0 & 1 & 0 & \\
& & 1 & &
\end{array}
\end{align}
 In this paper, we prove that if the rationally elliptic fourfold $X\subset\mathbb{P}^8$, then the Hodge conjecture also holds for this case.
 Thus we obtain the following result.

\begin{theorem}\label{thm:hodge}
The Hodge conjecture holds  for  fourfold $X\subset\mathbb{P}^8$  of elliptic homotopy type.
\end{theorem}

\subsection{The Hodge level}\label{subsec 1.3}

For a complex manifold $X$, the\emph{ Hodge level} of $H^{k}(X;\mathbb C)$ is defined to be the largest number $\mid q-p\mid$ such that $p+q=k$ and $h^{p,q}\neq 0$, and to be $-\infty$ if $H^{k}(X;\mathbb C)=0$ (see \cite{Ra}).

It is shown in \cite[Theorem 1.1,Theorem 1.3]{AB} that the Hodge level of a  compact K\"ahler manifold with elliptic homotopy type is $\leq 0$ in dimension $\leq 3$. For the  compact K\"ahler fourfold of elliptic homotopy type, the result in \cite[Proposition 4.2]{SY} implies  that the Hodge level is $\leq 0$, except for the case that  the even dimensional Hodge diamond is
\begin{align}\label{E hodge-level}
\begin{array}{ccccc}
& & 1 & & \\
& 1 & 2 & 1 & \\
0 & 2 & 2 & 2 & 0\\
& 1 & 2 & 1 & \\
& & 1 & &
\end{array}
\end{align}
(see case (g) in \cite[Theorem 1.3]{SY}).
We prove that there is no rationally elliptic  fourfold $X\subset\mathbb{P}^8$ in such case. Consequently, we have the following Theorem.

\begin{theorem}\label{thm:level} Let $X\subset\mathbb{P}^8$ be a  projective fourfold of elliptic homotopy type. Then the Hodge level of $X$ is $\leq 0$.
\end{theorem}

\subsection{Idea of proof}\label{subsec general-idea}
The idea for the proof of Theorems \ref{thm:RHPS},  \ref{thm:hodge} and \ref{thm:level} is sketched as follows.
Consider a simply-connected  projective fourfold $X\subset\mathbb{P}^8$ with elliptic homotopy type.
Under the assumptions on $X$ of Theorems \ref{thm:RHPS}, \ref{thm:hodge} and \ref{thm:level},  we first get all the possible Chern classes of $X$ by using some formulas of the Chern numbers.
The next key step is to eliminate all impossible cases of Chern classes case-by-case.
To achieve this, we first compute the Euler characteristic of the normal bundle $N_{X}\subset i^{\ast}T_{\mathbb{P}^8}$ for a holomorphic embedding $i:X \to \mathbb{P}^8$ by the following two methods:
\begin{itemize}
\item[(1)] according to the decomposition of complex vector bundle $i^{\ast}T_{\mathbb{P}^{8}}=T_{X}\oplus N_{X}$, we have the relation   $i^{\ast}c(\mathbb{P}^8)=c(X)\cdot c(N_{X})$, and consequently the Euler characteristic
    \[
    \chi(N_{X})=\langle c_{4}(N_{X}),[X] \rangle
    \]
     follows from the known Chern classes $i^{\ast}c(\mathbb{P}^8), c(X)$,
\item[(2)] from self-intersection formula  (c.f.\cite{HA},\cite{RA}),  we have that  the Euler characteristic
\[
\chi(N_{X})=\lambda(i,i),
\]
where $\lambda(i,i)$ is the self-intersection number of $i_{\ast}[X]$ in $\mathbb P^8$.
\end{itemize}
and then by comparing the resulting formulas, we eliminate all impossible cases of Chern classes of $X$.

\subsection{Structure of this paper}
We  apply the idea sketched in \S \ref{subsec general-idea} to prove  Theorem \ref{thm:RHPS}, Theorem \ref{thm:hodge} and Theorem \ref{thm:level}  in \S \ref{sec:2}, \S \ref{sec:3} and \S \ref{sec:4} respectively.
Moreover, some technical results on the polynomial equations with integral coefficients needed in the proof are presented in the appendix \S \ref{sec:5}.

\

\noindent \emph{Acknowledgement.}
We would like to thank Cheng-Yong Du,  Yang Su and Song Yang for helpful communications. We also would like to thank Sai-Kee Yeung for providing us with some useful references about the fake projective spaces.

\section{Proof of  Theorem \ref{thm:RHPS}}\label{sec:2}

 First of all, we notice that the first Chern class $c_{1}(X)\neq 0$ for a rationally elliptic projective fourfold $X$ (c.f.\cite[Corollary 1.4]{SY}).
Since the second Betti number $b_{2}(X)=1$ of a $\mathbb Q$-homological projective four-space $X$,
so either $c_{1}(X)<0$ or $c_{1}(X)>0$.
To finish the proof of Theorem \ref{thm:RHPS},
we need to rule out the case $c_{1}(X)<0$.

\begin{lemma}\label{lem:RHPSB}
There does not exist a simply-connected $\mathbb Q$-homological   projective four-space $X\subset\mathbb{P}^8$ with $c_{1}(X)<0$.
\end{lemma}

\begin{proof} We follow the general idea in \S \ref{subsec general-idea} and prove this lemma by contradiction. The derivation of contradiction consists of two steps.

\vskip 0.08cm
\noindent\textbf{Step 1. The possible Chern classes for $c_{1}(X)<0$. }

Since the Hodge structure of $X$ is the same as $\mathbb{P}^4$, we may conclude that
\[
\langle c_{4},[X] \rangle=\chi(X)=5,\quad \chi(X,\mathscr{O}_{X})=h^{0}(X,\mathscr{O}_{X})=1.
\]
Then from the following formula (c.f.\cite{Hirz60})
\[
4\chi(X,\mathscr{O}_{X})-\chi^{1}(X)= \frac{1}{12}\langle(2c_{4}+c_{1}c_{3}),[X]\rangle
\]
we get $\langle c_{1}c_{3},[X] \rangle=50$, where
$
\chi^1(X)=\sum\limits_{q=0}^{n}(-1)^{q}h^{1,q}.
$

On the other hand, we have the following formula (c.f.\cite{Hirz})
\[
\chi(X,\mathscr{O}_{X})=\frac{1}{720}\langle (-c_{4}+c_{3}c_{1}+3c_{2}^{2}+4c_{2}c_{1}^{2} -c_{1}^{4}),[X]\rangle,
\]
which implies
$\langle (3c_{2}^{2}+4c_{2}c_{1}^{2}-c_{1}^{4}),[X] \rangle=725-\langle c_{3}c_{1},[X]\rangle =675.$

Since the fourth Betti number $b_{4}(X)=1$, we have that
\begin{align}\label{E c2=kc12}
c_{2}(X) = k\cdot c_{1}^{2}(X),\quad\text{ for a }\quad k\in \mathbb Q.
\end{align}
This shows that
\begin{align}\label{E 2.2}
(3k^{2}+4k-1)\langle c_{1}^{4},[X]\rangle=675.
\end{align}
On the other hand, we have the Miyaoka-Yau inequality
$ \displaystyle
\frac{5}{2}\langle c_{1}^{2}c_{2},[X] \rangle\geq \langle c_{1}^{4},[X]\rangle,
$
which together with \eqref{E c2=kc12} implies that $k\geq \frac{2}{5}$. Moreover, when $k=\frac{2}{5}$, the equality in the Miyaoka-Yau inequality holds, and then the universal covering space of $X$ is a ball (c.f.\cite{YST}), which contradicts the assumption that $X$ is simply connected. So we must have
\begin{align}\label{E 2.3}
k>\frac{2}5
\end{align}
and consequently
\begin{align}\label{E 2.4}
\frac{27}{25}\langle c_{1}^{4},[X] \rangle < 675.
\end{align}

Let $g$ be the positive generator of $H^{2}(X)$.
Define the {\bf degree}  of $X$ to be
\begin{align}\label{E degree-def}
d:=\langle g^{4},[X]\rangle .
\end{align}
Assume that $x$ be the generator of $ H^{4}(X)$ and $g^2=ex$ with $e\in \mathbb Z_{>0}.$
 Then we have that $d=e^{2}$. Now since the first Chern class $c_1(X)<0$, we can assume that $c_{1}(X)=rg$ for a $r\in \mathbb Z_{<0}$. Then since $e,r\in\mathbb Z$, by \eqref{E 2.4}, we have
\begin{align}\label{E 2.6}
0\leq e\leq 25,\quad -5\leq r< 0.
\end{align}

Suppose that  $k=\frac{n}{l}$ with $n, l \in \mathbb Z_{>0}$ and $\gcd(n,l)=1$. By \eqref{E c2=kc12}, we may conclude that
\begin{align}\label{E 2.7}
l \mid er^{2}.
\end{align}
Then \eqref{E 2.2}, \eqref{E 2.3}, \eqref{E 2.6}, \eqref{E 2.7} imply that
 \[
 e=15,r=-1,k=\frac{2}{3}.
 \]
 Thus by the following relation
\[
c_{1}(X)=rg,\quad c_{2}(X)=kr^{2}g^{2}, \quad c_{3}(X)=\frac{50}{rd}g^{3},\quad  c_{4}(X)=5
\]
we obtain  the Chern classes of $X$ are
\[
c_{1}(X)=-g,\quad c_{2}(X)=\frac{2}{3}g^{2}, \quad c_{3}(X)=-\frac{2}{9}g^{3},\quad  c_{4}(X)=5
\]

\vskip 0.08cm
\noindent\textbf{Step 2. Eliminating  the possible Chern classes.}

Let $t$ be the positive generator of $H^{2}(\mathbb{P}^8)$. Then for a holomorphic embedding
\[
i:X \to \mathbb{P}^8,
\]
 there is a $m\in \mathbb Z_{>0}$ such that $i^{\ast}t=mg$. This shows that
\begin{align*}
&i^{\ast}c_{1}(\mathbb{P}^{8})=i^{\ast}(9t)=9mg,\\
&i^{\ast}c_{2}(\mathbb{P}^{8})=i^{\ast}(36t^{2})=36m^{2}g^{2},\\
&i^{\ast}c_{3}(\mathbb{P}^{8})=i^{\ast}(84t^{3})=84m^{3}g^{3},\\
& i^{\ast}c_{4}(\mathbb{P}^{8})=i^{\ast}(126t^{4})=126m^{4}g^{4}=126m^4d.
\end{align*}
According to  $i^{\ast}c(\mathbb{P}^8)=c(X)\cdot c(N_{X})$,  the Euler characteristic of normal bundle $N_X$ of $i$ is
 \[\langle
 c_{4}(N_{X}),[X]\rangle=28350m^{4}+18900m^{3}+2700m^2-225m-30.
 \]
By the self-intersection formula, we have that the Euler characteristic of normal bundle
 \[
 \chi(N_{X})=\lambda(i,i)=\langle(i_{\ast}[X])^{2},\epsilon \rangle=d^{2}m^{8}=15^{4}m^{8}=50625m^{8}.
 \]
Therefore, the positive integer $m$ must satisfies the following equation
\[
28350m^{4}+18900m^{3}+2700m^2-225m-30=50625m^{8}.
\]
In appendix \S \ref{sec:5},  we will show in Lemma \ref{lem:1} that there is no positive integer solution to this equation. Therefore, the result holds.
\end{proof}

 Now we are in the core of the proof of Theorem \ref{thm:RHPS}.

\begin{lemma}\label{lem:RHPSS}
A  $\mathbb Q$-homological projective four-space $X$ with $c_{1}(X)>0$ is biholomorphic to $\mathbb{P}^{4}$.
\end{lemma}

\begin{proof}
We follow the same notations as in proof of Lemma \ref{lem:RHPSB}.
Since the first Chern class $c_{1}(X)>0$,
we may suppose that $c_{1}(X)=rg$ with $r\in \mathbb Z_{>0}$.
Note that for a $n$-dimensional Fano manifold $X$, the following  inequality holds (c.f.\cite{KMM}):
$\displaystyle c_{1}(X)^{n}\leq (\frac{n+2}{2})^{2n}.$
Then we may conclude that $\displaystyle r^{4}e^{2}\leq (\frac{4+2}{2})^{8}=3^{8}$.
This shows that
\begin{align}\label{E 2.8}
0\leq e\leq 81.
\end{align}
Recall that
\begin{align}\label{E 2.9}
l \mid er^{2}
\end{align}
for $k=\frac{n}{l}$, $n, l \in \mathbb Z$ and $\gcd(n,l)=1$.
It is well known  that the index $r$ of a Fano manifold satisfies (c.f.\cite{Mi80})
\begin{align}\label{E 2.10}
0< r\leq 5.
\end{align}
By \eqref{E 2.8},\eqref{E 2.9} and \eqref{E 2.10}, the relation $ (3k^{2}+4k-1)\langle c_{1}^{4},[X] \rangle=675$  implies that
\begin{enumerate}
\item $e=15,r=1,k=\frac{2}{3}$, or
\item $e=15,r=1,k=-2$, or
\item $e=25,r=1,k=\frac{2}{5}$, or
\item $e=40,r=1,k=-\frac{13}{8}$, or
\item $e=60,r=1,k=\frac{1}{4}$, or
\item $e=60,r=1,k=-\frac{19}{12}$, or
\item $e=10,r=2,k=-\frac{13}{8}$, or
\item $e=15,r=2,k=\frac{1}{4}$, or
\item $e=15, r=2,k=-\frac{19}{12}$, or
\item $e=1, r=5,k=\frac{2}{5}$.
\end{enumerate}

If the index $r=1,$ by \cite[Theorem 9.1.6 (v)]{PS},  the  degree $d=2,4,5.$ Then from $d=e^{2}$ we get $e=2$.  This contradicts the cases (1)-(6).
If the index $r=2,$  by \cite[Theorem 5.2.3 (i)]{PS}, the   degree $d\leq 22$. This shows that $e\leq 4$. This contradicts the cases (7)-(9).
Therefore, we may conclude that the index $r=5$. It is shown in \cite{KO} that for this case, $X$ is biholomorphic to $\mathbb{P}^{4}$.
\end{proof}

\begin{proof}[Proof of Theorem \ref{thm:RHPS}]
From the Lemma \ref{lem:RHPSB} and Lemma \ref{lem:RHPSS}, the Theorem \ref{thm:RHPS} holds.
\end{proof}

\begin{remark}
In fact, the cohomology ring $H^{\ast}(X)$ has torsion elements.
However, the above argument is to discuss the relationship between Chern numbers and Euler characteristic. Thus the torsion elements can be neglected.
\end{remark}

\section{Proof of  Theorem \ref{thm:hodge}} \label{sec:3}

This section is devoted to prove Theorem \ref{thm:hodge}.
By the discussion in \S \ref{subsec 1.2}, it suffices to consider the rationally elliptic fourfold with the even dimensional Hodge diamond being \eqref{E hodge-diamond-1}.
Recall that the first Chern class $c_{1}(X)\neq 0$ for a rationally elliptic fourfold $X$.
By the Hodge diamond \eqref{E hodge-diamond-1}, we may conclude that the first Chern class of $X$ is either $c_{1}(X)>0$ or $c_{1}(X)< 0$.
It is well known that the Hodge Conjecture \ref{Conj} holds for a Fano fourfold ($c_{1}(X)>0$) (c.f.\cite{CM}). So we only need to study the case $c_{1}(X)< 0$.
For a complex projective manifold $X$ of dimension $n$, if the Hodge Conjecture holds for Hodge classes of degree $p$, for all $p< n,$ then the Hodge Conjecture holds for Hodge classes of degree $n-p$ (c.f.\cite{LG}).  Thus  for a projective fourfold,  the Lefschetz (1,1)-classes Theorem (c.f.\cite{GH}) implies that the Hodge Conjecture on $Hdg^{2}(X)$ and $Hdg^{6}(X)$.
As a consequence, we only need to show that the Hodge classes in $Hdg^{4}(X)$ is algebraic.
To this end, we show that every Hodge classes in $Hdg^{4}(X)$ is a rational linear combination of Chern classes of some holomorphic vector bundles over $X$.

\begin{lemma}\label{lem:hodgeconj}
Let $X\subset\mathbb{P}^8$ be a rationally elliptic fourfold  with even dimensional Hodge diamond \eqref{E hodge-diamond-1}. If the first Chern class $c_{1}(X)<0$, then the Chern classes $c_{1}^{2}(X)$ and $c_{2}(X)$ are the generators of $Hdg^{4}(X)$.
\end{lemma}

\begin{proof}
The proof follows the same line as that of Lemma \ref{lem:RHPSB}.
We prove this lemma by contradiction.
So we first assume that $H^{4}(X;\mathbb Q)$ does not generated by the Chern classes $c_{1}^{2}(X)$ and $c_{2}(X)$.

\vskip 0.08cm

\noindent\textbf{Step 1. The possible Chern classes of $X$. }

According to the even dimensional Hodge diamond \eqref{E hodge-diamond-1}, we have
$Hdg^{4}(X)=H^{4}(X;\mathbb Q)$.
As we assume that the Chern classes $c_{1}^{2}(X)$ and $c_{2}(X)$ do not generate $H^{4}(X;\mathbb Q)$, and the second Betti number $b_{2}(X)=2$,  we have
\begin{equation}\label{E 3.1}
c_{2}(X) = k\cdot c_{1}^{2}(X),
\end{equation}
for some $k\in \mathbb Q$.
Note that the odd Betti numbers of rationally elliptic fourfold are all zero.
By the Hodge index Theorem,
the signature $\sigma(X)=2.$ This shows that the intersection form on $H^{4}(M)$ is positive-definite. Moreover, Kneser showed that such a matrix of the intersection form is congruent to identity matrix (c.f.\cite{KK}). Let $g$ be the positive generator of $H^{2}(X)$.  Since $g^{2}\in H^{4}(X)$, then  there exist  generators $x,y$ of $H^{4}(X)$ such that $g^{2}=ax+by$ with coefficients $a,b\in \mathbb Z_{>0}$
and satisfy $
\langle x^{2},[X] \rangle=\langle y^{2},[X]\rangle=1$ and  $\langle xy,[X]\rangle=0$.
Thus we have the degree of $X$,
$d=\langle g^{4},[X]\rangle=\langle (ax+by)^{2},[X]\rangle=a^{2}+b^{2}.$
The Hodge structure of $X$ implies that
\[
\langle c_{4},[X] \rangle=\chi(X)=6,\quad \chi(X,\mathscr{O}_{X})=h^{0}(X,\mathscr{O}_{X})=1.
\]
Using the same argument as the proof of Lemma \ref{lem:RHPSB}, we have
\[
\langle c_{1}c_{3},[X] \rangle=48, \quad  (3k^{2}+4k-1)\langle c_{1}^{4},[X] \rangle=678,
\]
and
\begin{align}\label{E 3.2}
k>\frac{2}{5}.
\end{align}
This implies that
$ \displaystyle
\frac{27}{25}\langle c_{1}^{4},[X] \rangle  < 678.
$
Since the first Chern class $c_1(X)<0$, we can suppose that $c_{1}(X)=rg$ for a $r\in \mathbb Z_{<0}$. By
\[
\langle c_{1}^{4},[X] \rangle=\langle r^4g^4,[X]\rangle=r^{4}d=r^{4}(a^{2}+b^{2}) <627,
\]
together with $r,d\in\mathbb Z$, we have
\begin{align}\label{E 3.3}
0\leq a,b\leq 25,\quad -5\leq r< 0.
\end{align}

Again, assume that  $k=\frac{n}{l}$ with $n, l \in \mathbb Z_{>0}$ and $\gcd(n,l)=1$. By \eqref{E 3.1}, we may conclude that
\begin{align}\label{E 3.4}
l \mid ar^{2},\quad \text{and}\quad l \mid br^{2}.
\end{align}
From  \eqref{E 3.2}, \eqref{E 3.3} and \eqref{E 3.4}, the relation $(3k^{2}+4k-1)\langle c_{1}^{4},[X] \rangle=678$  imply that
\begin{enumerate}
\item $a= 1,b= 0,r=-3,k=\frac{11}{9}$, or
\item $a= 1,b= 1,r=-1,k=10$, or
\item $a= 1,b= 1,r=-4,k=\frac{7}{16}$, or
\item $a= 4,b= 4,r=-2,k=\frac{7}{16}$, or
\item $a= 7,b= 8,r=-1,k=1$, or
\item $a= 8, b= 7,r=-1,k=1$, or
\item $a= 9,b= 0,r=-1,k=\frac{11}{9}$, or
\item $a= 3,b= 3,r=-2,k=\frac{7}{12}$, or
\item $a= 12,b= 12,r=-1,k=\frac{7}{12}$,  or
\item $a= 16, b= 16,r=-1,k=\frac{7}{16}$.
\end{enumerate}

Then by the following formula
\[
c_{1}(X)=rg,\quad c_{2}(X)=kr^{2}g^{2}, \quad c_{3}(X)=\frac{48}{rd}g^{3},\quad c_{4}(X)=6
\]
together with $d=a^2+b^2$, we  obtain  all the possible Chern numbers of $X$ for the cases (1)-(10).
\[
\begin{array}{c|c|c|c|c|c}

cases & \langle c_{1}^{4},[X]\rangle & \langle c_{1}c_{3},[X] \rangle & \langle c_{1}^{2}c_{2},[X] \rangle & \langle c_{2}^{2},[X] \rangle  & \langle c_{4},[X] \rangle  \\
\hline
(1)& 81 & 48& 99 & 121  & 6 \\
\hline
(2) & 2 & 48 & 20 & 200  &6 \\
\hline
(3) & 512 & 48  &  224  & 98  & 6 \\
\hline
(4) & 512 & 48  &  224  & 98  & 6 \\
\hline
(5) & 113 & 48  &  113 &             113 & 6 \\
\hline
(6) &  113 &  48  &  113  & 113 &  6 \\
\hline
(7)& 81 &  48  &  99  & 121 &  6\\
\hline
(8)& 288 &  48  &  168  & 98 &  6\\
\hline
(9)&  288  & 48  &  168  & 98 &  6\\
\hline
(10)& 512 &  48  &  224   & 98 &  6
\end{array}
\]

\noindent\textbf{Step 2. Eliminating  impossible Chern classes. }

First of all by the relation (c.f.\cite{Hirz60})
\[
\langle (c_{1}^{2}c_{2}+2c_{1}^{4}),[X]\rangle\equiv 0\  \pmod {12},
\]
the cases $(1), (5), (6), (7)$ are impossible. We next consider the rest cases.

With the same notations as in proof of Lemma \ref{lem:RHPSB}, we consider a holomorphic embedding $i:X \to \mathbb{P}^8$.
According to  $i^{\ast}c(\mathbb{P}^8)=c(X)\cdot c(N_{X}),$ we get the  Chern number of normal bundle
\begin{align*}
\langle c_{4}(N_{X}),[X]\rangle&=126dm^4-84rdm^3- (36kr^2d-36r^2d)m^2
\\
&-(\frac{432}{r}-18kr^3d+9r^3d)m+
(k^2r^4d-3kr^4d+r^4d+90).
\end{align*}
On the other hand,  we have the Euler characteristic of normal bundle
\[
\chi(N_{X})=\lambda(i,i)=\langle(i_{\ast}[X])^{2},\epsilon \rangle=\langle[(mg)^{4}]^{2},\epsilon \rangle=d^{2}m^{8}.
\]
Therefore, the positive integer $m$ satisfies the following equation
\begin{align*}
&126dm^4-84rdm^3-(36kr^2d-36r^2d)m^2-(\frac{432}{r}-18kr^3d+9r^3d)m\\
& +(k^2r^4d-3kr^4d+r^4d+90)=d^{2}m^{8}.
\end{align*}
with $d=a^2+b^2$. We will show in Lemma \ref{lem:2} that there is no positive integer solution to this equation for $a,b,r,k$ given in cases (2),(3),(4),(8),(9),(10). This shows that all cases of Chern numbers of $N_X$ are impossible, and hence a contradiction to the assumption that $H^{4}(X;\mathbb Q)$ does not generated by the Chern classes $c_{1}^{2}(X)$ and $c_{2}(X)$. The Lemma follows.
\end{proof}

\begin{proof}[Proof of Theorem \ref{thm:hodge}]
By the analysis in the beginning of this section, Theorem \ref{thm:hodge} follows from Lemma \ref{lem:hodgeconj}.
\end{proof}

\section{Proof of Theorem \ref{thm:level}}\label{sec:4}

 In this section we prove Theorem \ref{thm:level}.
Recall that for the projective fourfold of elliptic homotopy type,
the only possible even dimensional Hodge diamond with positive Hodge level is \eqref{E hodge-level}.
Different from the even dimensional Hodge diamond \eqref{E hodge-diamond-1},
it is difficult to determine the groups of Hodge classes in the case \eqref{E hodge-level}.
Therefore, to prove Theorem \ref{thm:level}, we need the following preparations.

\begin{lemma}\label{lem:level} For a smooth projective fourfold $X$ with elliptic homotopy type and even dimensional Hodge diamond \eqref{E hodge-level}, the dimensions of  groups of Hodge classes satisfy
\[
hdg^{2}(X)=hdg^{4}(X)=hdg^{6}(X)=hdg^{8}(X)=1.
\]
\end{lemma}

\begin{proof}
 Since $X$ is projective, we may pick up a rational K\"{a}hler two-form $\omega$ on $X$.
Without loss of generality, we suppose $\omega \in H^{1,1}(X)\cap H^{2}(X).$
Let $x$ be a primitive cohomology class of $H^{2}(X;\mathbb R)$ such that  $\{\omega, x\}$ is a basis of $H^{1,1}(X)$.  By Hard Lefschetz Theorem, the $\omega^{2},\omega x$ are  generators of $H^{2,2}(X)$ and $\omega^{3}x=0.$ Since $Hdg^{4}(X)\subset H^{2,2}(X)$, then every Hodge class in $ Hdg^{4}(X)$ is linear combination with real
coefficients of $\omega^{2},\omega x$.

 Assume that $hdg^4(X)=2$.  Recall that the Hodge conjecture holds for this case (see \S \ref{subsec 1.2}). This gives that the Chow ring $CH^{2k}(X)=Hdg^{2k}(X)$.    Let $\{\omega^{2}, a\omega^{2}+b \omega x\} $ be a basis of $ Hdg^{4}(X)$ with $a,b\in \mathbb R$.   Then we have that
\[
a\omega^{4}=\omega^{2}(a\omega^{2}+b \omega x )\in Hgd^{8}(X)
\]
and so $a\in \mathbb Q$. This implies that  $a \omega^{2} \in  Hdg^{4}(X)$ and
\[
b \omega x=(a\omega^{2}+b \omega x)-a \omega^{2} \in Hdg^{4}(X).
\]
Thus $b \omega x$ corresponds  a rational linear combination  of cohomology classes of  algebraic cycles.

Note that $\omega$ corresponds an ample divisor on $X$ by $CH^{2}(X)=CH_{6}(X).$
Since
\[
\langle\omega^{2}\cdot (b\omega x),[X]\rangle=0,
\]
 by the Hodge index theorem (c.f.\cite[Appendix A, Theorem 5.2]{HA}), we have
$
 \langle (b\omega x)^{2},[X]\rangle>0.
$
This contradicts the Hodge--Riemann bilinear relation.
Indeed, from the Hodge--Riemann bilinear relation, we have  $\langle\omega^{2}x^{2},[X]\rangle<0$.
So we must have $hdg^4(X)=1$.
According to Hard Lefschetz Theorem, that $hdg^4(X)=1$ implies $hdg^{2}(X)=1$.  The Lemma follows.
\end{proof}

Now we prove the main result of this section.

\begin{lemma}\label{lem:levelzero}
 There is no rationally elliptic projective fourfold $X\subset\mathbb{P}^8$ with the even dimensional Hodge diamond \eqref{E hodge-level}.
\end{lemma}

\begin{proof}
 Let $X \subset\mathbb{P}^8$ be a  projective fourfold  with elliptic homotopy type and even  dimensional Hodge diamond \eqref{E hodge-level}. By Lemma \ref{lem:level}, we see that $X$ is monotone. It is shown in \cite[Corollary 1.4]{SY} that the first Chern class $c_{1}(X)<0.$  The same reason as in Lemma \ref{lem:RHPSB},  we prove the result in two steps.

\vskip 0.08cm
\noindent\textbf{Step 1. The possible Chern classes of $X$. }

Since $hdg^{4}(X)=1$, we have the relation
\begin{align}\label{E 4.1}
c_{2}(X) = k\cdot c_{1}^{2}(X),\quad\text{ for a }\quad k\in \mathbb Q.
\end{align}
Let $g$ be the positive generator of $Hdg^{2}(X)\cap H^{2}(X)$. The the first Chern class $c_{1}(X)=rg$ for a $r\in \mathbb Z_{<0}$.
Since the even dimensional Hodge diamond of $X$ is \eqref{E hodge-level}, together with all odd Betti numbers are zero, we have
$
\langle c_{4},[X] \rangle=\chi(X)=16$ and $\chi(X,\mathscr{O}_{X})= h^{0}(X,\mathscr{O}_{X})+h^{2}(X,\mathscr{O}_{X})=2$.

Using the same argument as the proof of Lemma \ref{lem:RHPSB}, we have
\[
\langle c_{1}c_{3},[X] \rangle=112, \quad  (3k^{2}+4k-1)\langle c_{1}^{4},[X] \rangle=1344,
\]
and
\begin{equation}\label{E 4.2}
k>\frac{2}{5}.
\end{equation}
 This implies that
$ \displaystyle
\frac{27}{25}\langle c_{1}^{4},[X] \rangle=\frac{27}{25}r^{4}d < 1344,
$
where $d$ is the degree of $X$ defined in \eqref{E degree-def}. Since $d,r\in\mathbb Z$, we have
\begin{align}\label{E 4.3}
0\leq d\leq 1244,\quad -5\leq r< 0.
\end{align}

Again, suppose  $k=\frac{n}{l}$ with $n, l \in \mathbb Z_{>0}$ and $\gcd(n,l)=1$, then by \eqref{E 4.1}, we get
\begin{align}\label{E 4.4}
l^{2} \mid dr^{4}.
\end{align}
By \eqref{E 4.2}, \eqref{E 4.3} and \eqref{E 4.4}, the relation $ (3k^{2}+4k-1)\langle c_{1}^{4},[X] \rangle=1344$  shows that
\begin{enumerate}
\item $d=3,r=-4,k=\frac{1}{2}$, or
\item $d=14,r=-2,k=1$, or
\item $d=48, r=-2, k=\frac{1}{2}$, or
\item $d=224, r=-1,k=1$, or
\item $d=768,r=-1,k=\frac{1}{2}$.
\end{enumerate}
Then by the following formulas
\[
c_{1}(X)=rg, \quad c_{2}(X)=kr^{2}g^{2}, \quad c_{3}(X)=\frac{112}{rd}g^{3},\quad c_{4}(X)=16
\]
we  obtain  all the possible Chern classes of $X$.

\vskip 0.08cm
\noindent\textbf{Step.2. Eliminating  the possible Chern classes.}

First of all,  the $\hat {A}$ genus of the case (2) is
\begin{align*}
 \hat {A}(X)=\frac{1}{5760}\langle(-4p_{2}+7p_{1}^{2}),[X]\rangle&=\frac{1}{5760}\langle(-4(k^{2}r^{4}d-192)+7(1-2k)^{2}r^{4}d),[X]\rangle\\
&=\frac{1}{5760}(-4(16\times 14-192)+7\times16 \times 14)=\frac{1}{4}
\end{align*}
This contradicts that  $\hat {A}$ genus of Spin manifold is integral.   We next consider the rest cases.

With the same notations as in proof of Lemma \ref{lem:RHPSB}, we consider a holomorphic embedding $i:X \to \mathbb{P}^8$.
According to  $i^{\ast}c(\mathbb{P}^8)=c(X)\cdot c(N_{X}),$ we get the  Chern number of normal bundle
\begin{align*}
\langle c_{4}(N_{X}),[X]\rangle&= 126dm^4-84rdm^3-(36kr^2d-36r^2d)m^2\\
&-(\frac{1008}{r}-18kr^3d+9r^3d)m+ (k^2r^{4}d-3kr^4d+r^4d+208).
\end{align*}
Then as the proof of Lemma \ref{lem:RHPSB}, the positive integer $m$ must satisfies the following equation
\begin{align*}
&126dm^4-84rdm^3-(36kr^2d-36r^2d)m^2-(\frac{1008}{r}-18kr^3d+9r^3d)m\\
&+(k^2r^{4}d-3kr^4d+r^4d+208)=d^{2}m^{8}.
\end{align*}
where $d,r,k$ are given in cases (1), (3), (4), (5). We will also show in Lemma \ref{lem:3} that the above equation on $m$ has no positive integer solutions for all the four cases. This finishes the proof.
\end{proof}

\begin{proof}[Proof of Theorem \ref{thm:level}]
By the analysis in the beginning of this section, Theorem \ref{thm:level} follows from Lemma \ref{lem:levelzero}.
\end{proof}


\appendix
\section{}\label{sec:5}

This appendix provides the technical details for the proof of our main theorems.  More pricisely, we show that there are no positive integer solution of the equations in the proof of   Lemma \ref{lem:RHPSB}, Lemma \ref{lem:hodgeconj}, Lemma \ref{lem:levelzero} respectively.

\begin{lemma}\label{lem:1}
There is no  positive integral solution of the following equation
$$
28350m^{4}+18900m^{3}+2700m^2-225m-30=50625m^{8}.
$$
\end{lemma}

\begin{proof}
The equation
\[
50625m^{8}-28350m^{4}-18900m^{3}-2700m^2+225m+30=0
 \]
 can be changed into
$
 3375m^{8}-1890m^{4}-1260m^{3}-180m^2+15m=-2.
$
Since $m$ is integral, by $3 \mid (3375m^{8}-1890m^{4}-1260m^{3}-180m^2+15m)$ and $3 \nmid -2$, we can get a contradiction.
\end{proof}

\begin{lemma}\label{lem:2}
For the cases $(2)$, $(3)$, $(4)$, $(8)$, $(9)$, $(10)$ in the proof of Lemma \ref{lem:hodgeconj},
there is no positive integral  solution of the following equation with unknown number $m$.
\begin{align*}
&126dm^4-84rdm^3-(36kr^2d-36r^2d)m^2-(\frac{432}{r}-18kr^3d+9r^3d)m\\
&+(k^2r^4d-3kr^4d+r^4d+90)=d^{2}m^{8}
\end{align*}
\end{lemma}

\begin{proof}  For the case $(2)$. If $\ a= 1,b= 1,r=-1,k=10,$ then the above equation is
$$
4m^{8}-252m^4-168m^3+648m^2-90m-232=0.$$
It can be changed into $2m^8-126m^4-84m^3+324m^2-45m=116$.
This implies that if $m$ positive integral, then $m \mid 116$ and so the probability of $m$ is $1,2,4,29,58,116$. Let
 \[
 f(m)=2m^8-126m^4-84m^3+324m^2-45m-116.
  \]
  It is shown in
  \[
  f(1)=-45\neq 0,f(2)=-1086\neq 0,f(4)=98328\neq 0,f(29)\doteq 1.0004 \times 10^{12}\neq 0,
   \]
   \[
   f(58)\doteq 2.5612 \times 10^{14} \neq 0, f(116)\doteq 6.5568 \times 10^{16} \neq 0
   \]
    that there is no positive integral solution of case (2).

 For the case (3). If  $\ a= 1,b= 1,r=-4,k=\frac{7}{16},$ then the above equation is
  \[
  4m^{8}-252m^4-672m^3-648m^2-252m-28=0.
   \]
   It can be changed into
   $
   m^{8}-63m^4-168m^3-162m^2-63m=7.
   $
    This implies that if $m$ positive integral, then $m \mid 7$ and so the probability of $m$ is $1,7$. Let $$f(m)=m^{8}-63m^4-168m^3-162m^2-63m-7.$$ It is shown in $f(1)=462\neq 0,f(7)=5547528\neq 0$ that there is no positive integral solution of case (3).

 For the case (4). If $\ a= 4,b= 4,r=-2,k=\frac{7}{16},$ then the above equation is
 \[
 1024m^{8}-4032m^{4}-5376m^{3}-2592m^2-504m-28=0.
 \]
 It can be changed into $256m^{8}-1008m^{4}-1344m^{3}-648m^2-126m=7.$ If $m$ is an integral, by $2 \mid (256m^{8}-1008m^{4}-1344m^{3}-648m^2-126m)$ and $2 \nmid 7$, we can get a contradiction.

 For the case (8). If $\ a= 3,b= 3,r=-2,k=\frac{7}{12},$  then the above equation is
 \[
 324m^{8}-2268m^{4}-3024m^{3}-1080m^2+28=0.
  \]
  It can be changed into $81m^{8}-567m^{4}-756m^{3}-270m^2=-7.$ If $m$ is an integral, by $9 \mid 81m^{8}-567m^{4}-756m^{3}-270m^2$ and $9 \nmid -7$, we can get a contradiction.

For the case (9). If $\ a= 12,b= 12,r=-1,k=\frac{7}{12},$ then the above equation is
 \[
 82944m^{8}-36288m^{4}-24192m^{3}-4320m^2+28=0.
 \]
 It can be changed into $20736m^{8}-9072m^{4}-6048m^{3}-1080m^2=-7.$
 If $m$ is an integral, by $2 \mid (20736m^{8}-9072m^{4}-6048m^{3}-1080m^2)$ and $2 \nmid -7$, we can get a contradiction.

For the case (10). If $\ a= 16, b= 16,r=-1,k=\frac{7}{16},$ then the above equation is
 \[
 262144m^{8}-64512m^{4}-43008m^{3}-10368m^2-1008m-28=0.
 \]
 It can be changed into $65536m^{8}-16128m^{4}-10752m^{3}-2592m^2-252m=7.$
 This implies that if $m$ positive integral, then $m \mid 7$ and so the probability of $m$ is $1,7$. Let
 \[
 f(m)=65536m^{8}-16128m^{4}-10752m^{3}-2592m^2-252m-7.
  \]
  It is shown in $f(1)= 35805\neq 0,f(7)\doteq 3.7776 \times 10^{11}\neq 0$ that there is no positive integral solution of case (10).
\end{proof}

\begin{lemma}\label{lem:3}
For the cases $(1), (3), (4), (5)$ in the proof of Lemma \ref{lem:levelzero},
there is no  positive integral solution of the following equation with unknown number $m$.
\begin{align*}
&126dm^4-84rdm^3-(36kr^2d-36r^2d)m^2-(\frac{1008}{r}-18kr^3d+9r^3d)m\\
&+(k^2r^{4}d-3kr^4d+r^4d+208)=d^{2}m^{8}.
\end{align*}
\end{lemma}

\begin{proof} For the case (1). If $\ d=3,r=-4,k=\frac{1}{2}, $ then the above equation is
\[
9m^8-378m^4-1008m^3-864m^2-252m=16.
\]
 If $m$ is integral, by $3 \mid (9m^8-378m^4-1008m^3-864m^2-252m)$ and $3 \nmid 16$, we can get a contradiction.

For the case (3). If $\ d=48, r=-2, k=\frac{1}{2},$  then the above equation is
\[
2304m^8-6048m^4-8064m^3-3456m^2-504m=16.
  \]
  If $m$ is integral, by $72 \mid (2304m^8-6048m^4-8064m^3-3456m^2-504m)$ and $72 \nmid 16$, we can get a contradiction.

For the case (4).  If  $\ d=224, r=-1,k=1,$ then the above equation is
 \[
 50176m^8-28224m^4-18816m^3+1008m+16=0.
   \]
   It can be changed into
   $
   3136m^{8}-1764m^{4}-1176m^{3}+63m=-1.
   $
    If $m$ is an integral, by $7 \mid (3136m^{8}-1764m^{4}-1176m^{3}+63m)$ and $7\nmid -1,$ we can get a contradiction.

For the case (5).  If $\ d=768,r=-1,k=\frac{1}{2},$ then the above equation is
\[
589824m^8-96768m^4-64512m^3-13824m^2-1008m-16=0.
\]
  It can be changed into
  $
  36864m^{8}-6048m^{4}-4032m^{3}-864m^2-63m=1.
  $
If $m$ is integral, by $9 \mid (36864m^{8}-6048m^{4}-4032m^{3}-864m^2-63m)$ and $9 \nmid 1$, we can get a contradiction.
\end{proof}


\end{document}